\documentclass[reqno,10pt]{article}
\usepackage{graphicx}
\usepackage{amsmath}
\usepackage{amssymb}
\usepackage{amsthm}
\usepackage{enumitem}
\usepackage{bbm}
\usepackage[l2tabu,orthodox]{nag}
\usepackage[all,error]{onlyamsmath}
\usepackage[strict]{csquotes}
\usepackage{url}

\newtheorem{theorem}{Theorem}
\newtheorem{lemma}[theorem]{Lemma}
\newtheorem{corollary}[theorem]{Corollary}
\newtheorem{proposition}[theorem]{Proposition}
\theoremstyle{definition}
\newtheorem{definition}[theorem]{Definition}
\newtheorem{remark}[theorem]{Remark}

\numberwithin{equation}{section}
\numberwithin{theorem}{section}

\newenvironment{OMabstract}{\noindent\textbf{Abstract.} }{\medskip}
\newenvironment{OMsubjclass}{\noindent\textbf{Mathematics Subject Classification (2020):} }{\medskip}
\newenvironment{OMkeywords}{\noindent\textbf{Keywords:}  }{\medskip}

\begin{document}

\author{Pablo Rocha} 
\title{Calder\'on-Hardy type spaces and the Heisenberg sub-Laplacian}
\maketitle


\begin{OMabstract}
For $0 < p \leq 1 < q < \infty$ and $\gamma > 0$, we introduce the Calder\'on-Hardy spaces $\mathcal{H}^{p}_{q, \gamma}(\mathbb{H}^{n})$ on the Heisenberg group $\mathbb{H}^{n}$, and show for every $f \in H^{p}(\mathbb{H}^{n})$ that the equation
\[
\mathcal{L} F = f
\]
has a unique solution $F$ in $\mathcal{H}^{p}_{q, 2}(\mathbb{H}^{n})$, where $\mathcal{L}$ is the sub-Laplacian on $\mathbb{H}^{n}$, 
$1 < q < \frac{n+1}{n}$ and $(2n+2) \, (2 + \frac{2n+2}{q})^{-1} < p \leq 1$.
\end{OMabstract}

\begin{OMkeywords}
Calder\'on-Hardy type spaces, Hardy type spaces, atomic decomposition, Heisenberg group, sub-Laplacian.
\end{OMkeywords}

\begin{OMsubjclass}
42B25, 42B30, 42B35, 43A80.
\end{OMsubjclass}

\section{Introduction}

The Laplace operator or Laplacian $\Delta$ on $\mathbb{R}^n$ is defined by
\[
\Delta = \sum_{j=1}^{n} \frac{\partial^2}{\partial x_j^2}.
\]
The ubiquity and the importance of this operator in physics and mathematics is well known. Needless to say that the study of problems involving the Laplacian are of interest either because of their applications or in their own right.

Given $m \in \mathbb{N}$, consider the inhomogeneous equation
\begin{equation} \label{delta problem}
\Delta^m F = f,
\end{equation}
where $\Delta^m$ is the iterated Laplacian, $f$ is a given data function and $F$ is an unknown function. Then, the problem consists in 
finding a function $F$ that solves (\ref{delta problem}) in some sense. It is common to address this problem by means of the fundamental solution of the operator $\Delta^m$. A \textit{fundamental solution} for $\Delta^m$ is a distribution $K$ on $\mathbb{R}^n$ such that 
$\Delta^m K = \delta$ in the distributional sense, where $\delta$ is Dirac's delta at the origin. In this case, for every $m \in \mathbb{N}$ fixed, we have that 
\[
\Phi_m(x) = \left\{ \begin{array}{cc}
                 C_1 \, |x|^{2m-n} \log{|x|}, & \text{if} \,\, n \,\, \text{is even and} \,\, 2m-n \geq 0  \\
                 C_2 \, |x|^{2m-n}, & \text{otherwise}
               \end{array} \right.
\]
is a fundamental solution for $\Delta^m$ on $\mathbb{R}^n$ (see p. 201-202 in \cite{Gelfand}). That fundamental solution is not uniquely determined. Indeed, $\Phi_m + u$ with $\Delta^m u = 0$, it is other fundamental solution for $\Delta^m$. These fundamental solutions are useful for producing solutions of the equation (\ref{delta problem}). For instance, if $m \geq 1$ and $f$ is a $C^{\infty}$-function with compact support, then $F= \Phi_m \ast f$ solves (\ref{delta problem}) in the classical sense. This formula also works for $m = 1$ when one assumes $f \in L^1(\mathbb{R}^n)$, and that $\int |f(x)| \log(|x|) dx < \infty$ in the case $n=2$, (see \cite[Theorem 2.21]{Gerald}). For $m \geq 1$ and $f \in L^p(\mathbb{R}^n)$ with $1 < p < \infty$, A. P. Calder\'on proved that there exists a locally integrable function $F$ what solves (\ref{delta problem}) in the distributional sense and $\| \partial^{\alpha} F \|_p \leq C \|f\|_p$ for all multi-index $\alpha$ such that $|\alpha| = 2m$, with $C$ independent of $f$ (see \cite[Lemma 8]{calderon}). 

It is known that the Hardy spaces $H^p(\mathbb{R}^n)$ are good substitutes for Lebesgue spaces 
$L^p(\mathbb{R}^n)$ when $0 < p \leq 1$ (see \cite{fefferman}, \cite{Elias}). In this direction, A. Gatto, J. Jim\'enez and C. Segovia in \cite{segovia}, posed the problem (\ref{delta problem}) for $m \geq 1$ and $f \in H^p(\mathbb{R}^n)$, $0 < p \leq 1$. To solve it they introduce the Calder\'on-Hardy spaces $\mathcal{H}^p_{q, \gamma}(\mathbb{R}^n)$, $0 < p \leq 1 < q < \infty$ and $\gamma > 0$, and proved for 
$n(2m + n/q)^{-1} < p \leq 1$ that given $f \in H^p(\mathbb{R}^n)$ there exists a unique $F \in \mathcal{H}^p_{q, 2m}(\mathbb{R}^n)$ that solves (\ref{delta problem}).

The underlying idea in \cite{segovia} to address this problem is the following: given $f \in H^p(\mathbb{R}^n)$, there exists an atomic decomposition $f = \sum k_j a_j$, such that $\|f\|_{H^p(\mathbb{R}^n)}^p \sim \sum k_j^p$ (see \cite{Latter}), then once defined the space 
$\mathcal{H}^p_{q, 2m}(\mathbb{R}^n)$ (which is defined as a quotient space) together with its "norm" 
$\| \cdot \|_{\mathcal{H}^p_{q, 2m}(\mathbb{R}^n)}$ , they define $b_j = (a_j \ast \Phi_m)$ and consider the class 
$B_j \in \mathcal{H}^p_{q, 2m}(\mathbb{R}^n)$ such that $b_j \in B_j$. Finally, for $n(2m+n/q)^{-1} < p \leq 1$, they prove that the series 
$\sum k_j B_j$ converges to $F$ in $\mathcal{H}^p_{q, 2m}(\mathbb{R}^n)$ and $\Delta^m F = f$. Moreover, $\Delta^m$ is a bijective mapping from $\mathcal{H}^p_{q, 2m}(\mathbb{R}^n)$ onto $H^p(\mathbb{R}^n)$, with 
$\| F \|_{\mathcal{H}^p_{q, 2m}(\mathbb{R}^n)} \sim \|\Delta^m F \|_{H^p(\mathbb{R}^n)}$.

In \cite{Duran}, R. Dur\'an extended the definition and atomic decomposition of $\mathcal{H}^p_{q, 2m}$ to the case of 
non-isotropic dilations on $\mathbb{R}^n$, solving an analogue problem to (\ref{delta problem}) for more general elliptic operators with symbols of the form $\xi_1^{2k_1} + \cdot \cdot \cdot + \xi_n^{2k_n}$, with $k_1, ..., k_n \in \mathbb{N}$.

The equation (\ref{delta problem}), for $f \in H^{p(\cdot)}(\mathbb{R}^{n})$ and $f \in H^{p}(\mathbb{R}^{n}, w)$, was studied by the present author in \cite{rocha1} and \cite{rocha2} respectively, obtaining analogous results to those of Gatto, Jim\'enez and Segovia.

Recently, Z. Liu, Z. He and H. Mo in \cite{He} extended the definition of Calder\'on-Hardy spaces to Orlicz setting. These new Orlicz 
Calder\'on-Hardy spaces can cover classical Calder\'on-Hardy spaces in \cite{segovia}. As an application, they solved the equation 
(\ref{delta problem}) when $f \in H^{\Phi}(\mathbb{R}^{n})$, where $H^{\Phi}(\mathbb{R}^{n})$ are the Orlicz-Hardy spaces defined 
in \cite{Nakai}.

On the other hand, it is well known that the Lie group "most commutative" among the non-commutative is the Heisenberg group, it plays an important role in several branches of mathematics (see \cite{Thangavelu}). So, one has the opportunity to ask whether certain standard 
results of Euclidean harmonic analysis can be adapted to the non-commutative setting of the Heisenberg group. Following this line, the 
purpose of this work is to pose and solve an analogous problem to (\ref{delta problem}) on the Heisenberg group with $m=1$. More precisely, for $f \in H^p(\mathbb{H}^n)$, $0 < p \leq 1$, we consider the equation
\begin{equation} \label{L problem}
\mathcal{L} F = f,
\end{equation}
where $\mathcal{L}$ is the sub-Laplacian on $\mathbb{H}^{n}$. The solution obtained in \cite{segovia}, for the Euclidean case, 
suggests us that once defined the space $\mathcal{H}^{p}_{q, 2}(\mathbb{H}^{n})$ a representative for the solution 
$F \in \mathcal{H}^{p}_{q, 2}(\mathbb{H}^{n})$ of (\ref{L problem}) should be $\sum k_j (a_j \ast_{\mathbb{H}^n} \Phi)$, where
$\sum k_j a_j$ is an atomic decomposition for $f \in H^p(\mathbb{H}^{n})$ (see \cite{Foll-St}), and $\Phi$ is the fundamental 
solution of $\mathcal{L}$ obtained by G. Folland in \cite{Folland}. We shall see that this argument works as well on $\mathbb{H}^n$, 
but taking into account certain non-trivial aspects inherent to the Heisenberg group.

Our main results are contained in Theorems \ref{principal thm} and \ref{2nd thm} (see Section 5 below). The first of them 
states that \textit{if $Q=2n+2$, $1 < q < \frac{n+1}{n}$ and $Q \, (2 + \frac{Q}{q})^{-1} < p \leq 1$, then the sub-Laplacian $\mathcal{L}$ 
on $\mathbb{H}^n$ is a bijective mapping from $\mathcal{H}^{p}_{q, 2}(\mathbb{H}^{n})$ onto $H^{p}(\mathbb{H}^{n})$. Moreover, for every 
$G \in \mathcal{H}^{p}_{q, 2}(\mathbb{H}^{n})$, the quantities $\| \mathcal{L}G \|_{H^{p}(\mathbb{H}^{n})}$ and
$\|G \|_{\mathcal{H}^{p}_{q, 2}(\mathbb{H}^{n})}$ are comparable with implicit constants independent of $G$.} In other words, for 
$Q \, (2 + \frac{Q}{q})^{-1} < p \leq 1$ and $f \in H^p(\mathbb{H}^n)$, the equation (\ref{L problem}) has a unique solution in 
$\mathcal{H}^{p}_{q, 2}(\mathbb{H}^{n})$.

A key technical result needed to get Theorem \ref{principal thm} is Proposition \ref{pointwise estimate} below. This establishes a pointwise inequality in $\mathbb{H}^n$ which can be inferred from Gatto, Jim\'enez and Segovia's approach, however its analogous in $\mathbb{R}^n$ is not explicitly stated in \cite{segovia}.

Although the fundamental solutions for the powers of the sub-Laplacian $\mathcal{L}^m$ are known for every integer $m \geq 2$ 
(see \cite{Benson}), the problem in this case is much more complicated. For this reason we focus solely on the case $m=1$.

Finally, our second result says that the case $0 < p \leq Q \, (2 + \frac{Q}{q})^{-1}$ is trivial. Indeed, we have that \textit{if 
$1 < q < \frac{n+1}{n}$ and $0 < p \leq Q \, (2 + \frac{Q}{q})^{-1}$, then $\mathcal{H}^{p}_{q, \, 2}(\mathbb{H}^{n}) = \{ 0 \}$.}

\

This paper is organized as follows. In Section 2 we state the basics of the Heisenberg group. The definition and atomic decomposition of Hardy spaces on the Heisenberg group are presented in Section 3. We introduce the Calder\'on-Hardy spaces on the Heisenberg group and investigate their properties in Section 4. The key technical result mentioned above is also stated in Section 4. Finally, our main results are proved in Section 5.

\

\textbf{Notation:} The symbol $A \lesssim B$ stands for the inequality $A \leq c B$ for some constant $c$. We denote by 
$B(z_0, \delta)$ the $\rho$ - ball centered at $z_0 \in \mathbb{H}^{n}$ with radius $\delta$. Given $\beta > 0$ and a $\rho$ - ball 
$B= B(z_0, \delta)$, we set $\beta B= B(z_0, \beta \delta)$. For a measurable subset $E\subseteq \mathbb{H}^{n}$ we denote by $\left\vert E\right\vert $ and $\chi_{E}$ the Haar measure of $E$ and the characteristic function of $E$ respectively. Given a real number $s \geq 0$, we write $\lfloor s \rfloor$ for the integer part of $s$.

Throughout this paper, $C$ will denote a positive constant, not necessarily the same at each occurrence.

\section{Preliminaries}

The  Heisenberg group $\mathbb{H}^{n}$ can be identified with $\mathbb{R}^{2n} \times \mathbb{R}$ whose group law 
(noncommutative) is given by
\[
(x,t) \cdot (y,s) = \left( x+y, t+s + x^{t} J y \right),
\]
where $J$ is the $2n \times 2n$ skew-symmetric matrix given by
\[
J= 2 \left( \begin{array}{cc}
                           0 & -I_n \\
                           I_n & 0 \\
                                      \end{array} \right)
\]
being $I_n$ the $n \times n$ identity matrix.

The dilation group on $\mathbb{H}^n$ is defined by 
\[
r \cdot (x,t) = (rx, r^{2}t), \,\,\,\ r > 0.
\]
With this structure we have that $e = (0,0)$ is the neutral element, $(x, t)^{-1}=(-x, -t)$ is the inverse of $(x, t)$, and
$r \cdot((x,t) \cdot (y,s)) = (r\cdot(x,t)) \cdot (r\cdot(y,s))$. 

The \textit{Koranyi norm} on $\mathbb{H}^{n}$ is the function $\rho : \mathbb{H}^{n} \to [0, \infty)$ defined by
\begin{equation} \label{Koranyi norm}
\rho(x,t) = \left( |x|^{4} + \, t^{2}  \right)^{1/4}, \,\,\, (x,t) \in \mathbb{H}^{n},
\end{equation}
where $| \cdot |$ is the usual Euclidean norm on $\mathbb{R}^{2n}$. It is easy to check that $|x| \leq \rho(x,t)$ and $|t| \leq \rho(x,t)^2$.

Let $z = (x,t)$ and $w = (y,s) \in \mathbb{H}^{n}$, the Koranyi norm satisfies the following properties:
\begin{equation*}
\begin{split}
\rho(z) & = 0 \,\,\, \text{if and only if} \,\, z = e, \\
\rho(z^{-1}) & = \rho(z) \,\,\,\, \text{for all} \,\, z \in \mathbb{H}^{n}, \\
\rho(r \cdot z) & = r \rho(z) \,\,\,\, \text{for all} \,\, z \in \mathbb{H}^{n} \,\, \text{and all} \,\, r > 0, \\
\rho(z \cdot w) & \leq  \rho(z) + \rho(w) \,\,\,\, \text{for all} \,\, z, w \in \mathbb{H}^{n}, \\
| \rho(z) - \rho(w) | & \leq \rho(z \cdot w) \,\,\,\, \text{for all} \,\, z, w \in \mathbb{H}^{n}.
\end{split}
\end{equation*}
Moreover, $\rho$ is continuous on $\mathbb{H}^{n}$ and is smooth on $\mathbb{H}^{n} \setminus \{ e \}$. The $\rho$ - ball centered at 
$z_0 \in \mathbb{H}^{n}$ with radius $\delta > 0$ is defined by
\[
B(z_0, \delta) := \{ w \in \mathbb{H}^{n} : \rho(z_0^{-1} \cdot w) < \delta \}.
\]

The topology in $\mathbb{H}^{n}$ induced by the $\rho$ - balls coincides with the Euclidean topology of 
$\mathbb{R}^{2n} \times \mathbb{R} \equiv\mathbb{R}^{2n+1}$ (see \cite[Proposition 3.1.37]{Fischer}). So, the borelian sets of 
$\mathbb{H}^{n}$ are identified with those of $\mathbb{R}^{2n+1}$. The Haar measure in $\mathbb{H}^{n}$ is the Lebesgue measure of 
$\mathbb{R}^{2n+1}$, thus $L^{p}(\mathbb{H}^{n}) \equiv L^{p}(\mathbb{R}^{2n+1})$, for every $0 < p \leq \infty$. Moreover, for 
$f \in L^{1}(\mathbb{H}^{n})$ and for $r > 0$ fixed, we have
\begin{equation} \label{homog dim}
\int_{\mathbb{H}^{n}} f(r \cdot z) \, dz = r^{-Q} \int_{\mathbb{H}^{n}} f(z) \, dz,
\end{equation}
where $Q= 2n+2$. The number $2n+2$ is known as the \textit{homogeneous dimension} of $\mathbb{H}^{n}$ (we observe that the \textit{topological dimension} of $\mathbb{H}^{n}$ is $2n+1$).

Let $|B(z_0, \delta)|$ be the Haar measure of the $\rho$ - ball $B(z_0, \delta) \subset \mathbb{H}^{n}$. Then, 
\[
|B(z_0, \delta)| = c \delta^{Q},
\]
where $c = |B(e,1)|$ and  $Q = 2n+2$. Given $\lambda > 0$, we put $\lambda B = \lambda B(z_0, \delta) = 
B(z_0, \lambda \delta)$. So $|\lambda B| = \lambda^{Q}	|B|$.

\begin{remark}
For any $z, z_0 \in \mathbb{H}^{n}$ and $\delta >0$, we have
\[
z_0 \cdot B(z, \delta) = B(z_0 \cdot z, \delta).
\]
In particular, $B(z, \delta) = z \cdot B(e, \delta)$. It is also easy to check that $B(e, \delta) = \delta \cdot B(e,1)$ for any $\delta > 0$.
\end{remark}
\begin{remark} \label{cambio de centro}
If $f \in L^{1}(\mathbb{H}^{n})$, then for every $\rho$ - ball $B$ and every $z_0 \in \mathbb{H}^{n}$, we have
\[
\int_{B} f(w) \, dw = \int_{z_{0}^{-1} \cdot B} f(z_0 \cdot u) \, du.
\]
\end{remark}

The Hardy-Littlewood maximal operator $M$ is defined by
\[
Mf(z) = \sup_{B \ni z} |B|^{-1}\int_{B} |f(w)| \, dw,
\]
where $f$ is a locally integrable function on $\mathbb{H}^{n}$ and the supremum is taken over all the $\rho$ - balls $B$ containing $z$.

\

If $f$ and $g$ are measurable functions on $\mathbb{H}^{n}$, their convolution $f * g$ is defined by
\[
(f * g)(z) := \int_{\mathbb{H}^{n}} f(w) g(w^{-1} \cdot z) \, dw,
\]
when the integral is finite.

For every $i = 1,2, ..., 2n+1$, $X_i$ denotes the left invariant vector field given by
\[
X_i = \frac{\partial}{\partial x_i} + 2 x_{i+n} \frac{\partial}{\partial t}, \,\,\,\, i=1, 2, ..., n;
\]
\[
X_{i+n} = \frac{\partial}{\partial x_{i+n}} - 2 x_{i} \frac{\partial}{\partial t}, \,\,\, i=1, 2, ..., n;
\]
and
\[
X_{2n+1} = \frac{\partial}{\partial t}.
\]
Similarly, we define the right invariant vector fields $\{ \widetilde{X}_i \}_{i=1}^{2n+1}$ by
\[
\widetilde{X}_i = \frac{\partial}{\partial x_i} - 2 x_{i+n} \frac{\partial}{\partial t}, \,\,\,\, i=1, 2, ..., n;
\]
\[
\widetilde{X}_{i+n} = \frac{\partial}{\partial x_{i+n}} + 2 x_{i} \frac{\partial}{\partial t}, \,\,\, i=1, 2, ..., n;
\]
and
\[
\widetilde{X}_{2n+1} = \frac{\partial}{\partial t}.
\]
The sub-Laplacian on $\mathbb{H}^n$, denoted by $\mathcal{L}$, is the counterpart of the Laplacain $\Delta$ on $\mathbb{R}^n$. The 
sub-Laplacian $\mathcal{L}$ is defined by
\[
\mathcal{L} =- \sum_{i=1}^{2n} X_i^2,
\]
where $X_i$, $i= 1, ..., 2n$, are the left invariant vector fields defined above.

Given a multi-index $I=(i_1,i_2, ..., i_{2n}, i_{2n+1}) \in (\mathbb{N} \cup \{ 0 \})^{2n+1}$, we set
\[
|I| = i_1 + i_2 + \cdot \cdot \cdot + i_{2n} + i_{2n+1}, \hspace{.5cm} d(I) = i_1 + i_2 + \cdot \cdot \cdot + i_{2n} + 2 \, i_{2n+1}.
\]
The amount $|I|$ is called the length of $I$ and $d(I)$ the homogeneous degree of $I$. We adopt the following multi-index notation for
higher order derivatives and for monomials on $\mathbb{H}^{n}$. If $I=(i_1, i_2, ..., i_{2n+1})$ is a multi-index, 
$X = \{ X_i \}_{i=1}^{2n+1}$, $\widetilde{X} =  \{ \widetilde{X}_{i} \}_{i=1}^{2n+1}$, and 
$z = (x,t) = (x_1, ..., x_{2n}, t) \in \mathbb{H}^{n}$, we put
\[
X^{I} := X_{1}^{i_1} X_{2}^{i_2} \cdot \cdot \cdot X_{2n+1}^{i_{2n+1}}, \,\,\,\,\,\, 
\widetilde{X}^{I} := \widetilde{X}_{1}^{i_1} \widetilde{X}_{2}^{i_2} \cdot \cdot \cdot \widetilde{X}_{2n+1}^{i_{2n+1}},
\]
and
\[
z^{I} := x_{1}^{i_1} \cdot \cdot \cdot x_{2n}^{i_{2n}} \cdot t^{i_{2n+1}}.
\]
A computation give
\[
X^{I}(f(r \cdot z)) = r^{d(I)} (X^{I}f)(r\cdot z), \,\,\,\,\,\, \widetilde{X}^{I}(f(r \cdot z)) = r^{d(I)} (\widetilde{X}^{I}f)(r\cdot z)
\]
and 
\[
(r\cdot z)^{I} = r^{d(I)} z^{I}.
\]
So, the operators $X^{I}$ and $\widetilde{X}^{I}$ and the monomials $z^{I}$ are homogeneous of degree $d(I)$. In particular, the 
sub-Laplacian $\mathcal{L}$ is an operator homogeneous of degree $2$. The operators $X^{I}$, $\widetilde{X}^{I}$, and $\mathcal{L}$  
interact with the convolutions in the following way
\[
X^{I}(f \ast g) = f \ast (X^{I}g), \,\,\,\,\,\, \widetilde{X}^{I}(f \ast g) = (\widetilde{X}^{I} f) \ast g, \,\,\,\,\,\,
(X^{I} f) \ast g = f \ast (\widetilde{X}^{I} g),
\]
and 
\[
\mathcal{L} (f \ast g) = f \ast \mathcal{L}g.
\]

Every polynomial $p$ on $\mathbb{H}^n$ can be written as a unique finite linear combination of the monomials $z^I$, that is
\begin{equation} \label{polynomial}
p(z) = \sum_{I \in \mathbb{N}_{0}^n} c_I z^I,
\end{equation}
where all but finitely many of the coefficients $c_I \in \mathbb{C}$ vanish. The \textit{homogeneous degree} of a polynomial $p$ written as 
(\ref{polynomial}) is $\max \{ d(I) : I \in \mathbb{N}_{0}^n \,\, \text{with} \,\, c_I \neq 0 \}$. Let $k \in \mathbb{N} \cup \{ 0 \}$, 
with $\mathcal{P}_{k}$ we denote the subspace formed by all the polynomials of homogeneous degree at most $k$. So, every 
$p \in \mathcal{P}_{k}$ can be written as $p(z) = \sum_{d(I) \leq k} c_I \, z^I$, with $c_I \in \mathbb{C}$.

\

The Schwartz space $\mathcal{S}(\mathbb{H}^{n})$ is defined as the collection of all the $\phi \in C^{\infty}(\mathbb{H}^{n})$ such that
\[
\sup_{z \in \mathbb{H}^{n}} (1+\rho(z))^{N} |(X^{I} \phi)(z)| < \infty,
\]
for all $N \in \mathbb{N}_{0}$ and all $I \in (\mathbb{N}_{0})^{2n+1}$. We topologize the space $\mathcal{S}(\mathbb{H}^{n})$ with the following family of semi-norms
\[
\| \phi \|_{\mathcal{S}(\mathbb{H}^{n}), \, N} = \sum_{d(I) \leq N} \sup_{z \in \mathbb{H}^{n}} (1+\rho(z))^{N} |(X^{I} \phi)(z)| \,\,\,\,\,\,\, (N \in \mathbb{N}_{0}),
\]
with $\mathcal{S}'(\mathbb{H}^{n})$ we denote the dual space of $\mathcal{S}(\mathbb{H}^{n})$.

\bigskip

A fundamental solution for the sub-Laplacian on $\mathbb{H}^n$ was obtained by G. Folland in \cite{Folland}. More precisely, he proved the following result.

\begin{theorem} \label{fund solution}
$c_n \, \rho^{-2n}$ is a fundamental solution for $\mathcal{L}$ with source at $0$, where
\[
\rho(x,t) = (|x|^4 + t^2)^{1/4},
\]
and
\[
c_n = \left[ n(n+2) \int_{\mathbb{H}^n} |x|^2 (\rho(x,t)^4 + 1)^{-(n+4)/2} dxdt \right]^{-1}.
\]
In others words, for any $u \in \mathcal{S}(\mathbb{H}^{n})$, $\left( \mathcal{L}u,  c_n \rho^{-2n} \right) = u(0)$.
\end{theorem}

\begin{lemma} \label{ro estimate} Let $\alpha > 0$ and $\rho(x,t) = (|x|^4 + t^2)^{1/4}$, then 
\[
\left| \widetilde{X}^{J} \left( X^{I} \rho^{-\alpha} \right)(x,t) \right| \leq C \rho(x,t)^{-\alpha - d(I) - d(J)},
\]
holds for all $(x,t) \neq e$ and every pair of multi-indixes $I$ and $J$.
\end{lemma}

\begin{proof}
The proof follows from the homogeneity of the kernel $\rho^{-\alpha}$, i.e.: 
$\rho(r \cdot (x,t))^{-\alpha} = r^{-\alpha} \rho(x,t)^{-\alpha}$, and from the homogeneity of the operators $\widetilde{X}^{J}$ and $X^{I}$.
\end{proof}

We conclude these preliminaries with the following supporting result.

\begin{lemma} \label{finite measure}
Let $0 < p < \infty$ and let $\mathcal{O}$ be a measurable set of $\mathbb{H}^n$ such that $|\mathcal{O}| < \infty$. If 
$h \in L^p(\mathbb{H}^n \setminus \mathcal{O})$, then
\[
|\{ z : |h(z)| < \epsilon\}| > 0, \,\,\,\,\,\, \text{for all} \,\,\, \epsilon > 0.
\]
\end{lemma}

\begin{proof}
Suppose that there exists $\epsilon_0 > 0$ such that $|\{ z : |h(z)| < \epsilon_0 \}| = 0$, so $|h(z)| \geq \epsilon_{0}/2$ \,
a.e. $z \in \mathbb{H}^n$, which implies that
\[
\infty = |\mathcal{O}^c| = |\{ z \in \mathcal{O}^c : |h(z)| \geq \epsilon_{0}/2 \}| \leq (2/\epsilon_{0})^p \| h \|_{L^p(\mathcal{O}^c)}^{p},
\]
contradicting the assumption that $h \in L^p(\mathbb{H}^n \setminus \mathcal{O})$. Then, the lemma follows.
\end{proof}

\section{Hardy spaces on the Heisenberg group}

In this section, we briefly recall the definition and the atomic decomposition of the Hardy spaces on the Heisenberg group 
(see \cite{Foll-St}).

Given $N \in \mathbb{N}$, define 
\[
\mathcal{F}_{N}=\left\{ \varphi \in \mathcal{S}(\mathbb{H}^{n}) : \| \varphi \|_{\mathcal{S}(\mathbb{H}^{n}), \, N} \leq 1 \right\}.
\]
For any $f \in \mathcal{S}'(\mathbb{H}^{n})$, the grand maximal function of $f$ is defined by 
\[
\mathcal{M}_N f(z)=\sup\limits_{t>0}\sup\limits_{\varphi \in \mathcal{F}_{N}}\left\vert \left( f \ast \varphi_t \right)(z) \right\vert,
\]
where $\varphi_t(z) = t^{-2n-2} \varphi(t^{-1} \cdot z)$ with $t > 0$. 

We put
\begin{equation}
N_p = \left\{ \begin{array}{cc}
                 \lfloor Q(p^{-1}-1) \rfloor + 1, & \text{if} \,\, 0 < p \leq 1  \\
                 0, & \text{if} \,\,\,\, 1 < p \leq \infty
               \end{array} \right..
\end{equation}
The Hardy space $H^{p}(\mathbb{H}^{n})$ is the set of all $f \in S^{\prime}
(\mathbb{H}^{n})$ for which $\mathcal{M}_{N_p}f \in L^{p}(\mathbb{H}^{n})$. In this case we define 
$\left\Vert f \right\Vert _{H^{p}(\mathbb{H}^{n})} = \left\Vert \mathcal{M}_{N_p}f \right\Vert_{L^{p}(\mathbb{H}^{n})}$. For $p > 1$, it is well known that $H^p(\mathbb{H}^{n}) \equiv L^p(\mathbb{H}^{n})$ and for $p=1$, $H^1(\mathbb{H}^{n}) \subset L^1(\mathbb{H}^{n})$. On the range $0 < p < 1$, the spaces $H^{p}(\mathbb{H}^{n})$ and $L^{p}(\mathbb{H}^{n})$ are not comparable.

Now, we introduce the definition of atom in $\mathbb{H}^n$.

\begin{definition}
Let $0 < p \leq 1 < p_{0} \leq \infty$. Fix an integer $N \geq N_{p}$. A measurable function $a(\cdot)$ on $\mathbb{H}^{n}$ is called an 
$(p, p_{0}, N)$ - atom if there exists a $\rho$ - ball $B$ such that \newline
$a_{1})$ $\textit{supp}\left( a\right) \subset B$, \newline
$a_{2})$ $\left\Vert a \right\Vert_{L^{p_{0}}(\mathbb{H}^{n})} \leq 
\left\vert B \right\vert^{\frac{1}{p_{0}} - \frac{1}{p}}$, \newline
$a_{3})$ $\int a(z) \, z^{I} \, dz = 0$ for all multiindex $I$ such that $d(I) \leq N$.
\end{definition}
A such atom is also called an atom centered at the $\rho$ - ball $B$. We observe that every $(p, p_{0}, N)$ - atom 
$a(\cdot)$ belongs to $H^{p}(\mathbb{H}^{n})$. Moreover, there exists an universal constant $C > 0$ such that $\| a \|_{H^p(\mathbb{H}^n)} \leq C$ for all $(p, p_{0}, N)$ - atom $a(\cdot)$.

\begin{remark} \label{atomo trasladado}
It is easy to check that if $a(\cdot)$ is a $(p, p_{0}, N)$ - atom centered at the $\rho$ - ball $B(z_0, \delta)$, then the function 
$a_{z_0}(\cdot) := a(z_0 \cdot (\cdot))$ is a $(p, p_{0}, N)$ - atom centered at the $\rho$ - ball $B(e, \delta)$.
\end{remark}

\begin{definition} Let $0 < p \leq 1 < p_{0} \leq \infty$ and let $N \geq N_{p}$ be fixed. The space 
$H_{atom}^{p, p_{0}, N}\left( \mathbb{H}^{n}\right)$ is the set of all distributions $f \in \mathcal{S}'(\mathbb{H}^{n})$ such that it can be
written as
\begin{equation}
f=\sum\limits_{j=1}^{\infty }k_{j}a_{j}  \label{desc. atomica}
\end{equation}
in $\mathcal{S}'(\mathbb{H}^{n}),$ where $\left\{ k_{j}\right\}_{j=1}^{\infty }$ is a sequence of non negative numbers, the $a_{j}$'s are 
$(p, p_{0}, N)$ - atoms and $\sum_j k_j^p < \infty$. Then, one defines
\[
\left\Vert f\right\Vert_{H_{atom}^{p, p_{0}, N}(\mathbb{H}^n)} :=\inf \left\{ \sum_j k_j^p : f=\sum\limits_{j=1}^{\infty }k_{j}a_{j} \right\}
\]
where the infimum is taken over all admissible expressions as in (\ref{desc. atomica}).
\end{definition}

For $0 < p \leq 1 < p_{0} \leq \infty$ and $N \geq N_{p}$, Theorem 3.30 in \cite{Foll-St} asserts that
\[
H_{atom}^{p, p_{0}, N}(\mathbb{H}^n) = H^p(\mathbb{H}^n)
\] 
and the quantities $\left\Vert f\right\Vert_{H_{atom}^{p, p_{0}, N}(\mathbb{H}^n)}$ and 
$\left\Vert f\right\Vert _{H^{p}(\mathbb{H}^n)}$ are comparable. Moreover, if $f \in H^p(\mathbb{H}^n)$ then admits an atomic decomposition 
$f=\sum\limits_{j=1}^{\infty }k_{j}a_{j}$ such that 
\[
\sum_j k_j^p \leq C \, \|f \|_{H^{p}(\mathbb{H}^n)}^p,
\]
where $C$ does not depend on $f$.

\section{Calder\'on-Hardy spaces on the Heisenberg group}

Let $L^{q}_{loc}(\mathbb{H}^{n})$, $1 < q < \infty$, be the space of all measurable functions $g$ on $\mathbb{H}^{n}$ that belong locally to 
$L^{q}$ for compact sets of $\mathbb{H}^{n}$. We endowed $L^{q}_{loc}(\mathbb{H}^{n})$ with the topology generated by the seminorms
\[
|g|_{q, \, B} = \left( |B|^{-1} \int_{B} \, |g(w)|^{q}\, dw \right)^{1/q},
\] 
where $B$ is a $\rho$-ball in $\mathbb{H}^{n}$ and $|B|$ denotes its Haar measure.

For $g \in L^{q}_{loc}(\mathbb{H}^{n})$, we define a maximal function $\eta_{q, \, \gamma}(g; z)$ as
\[
\eta_{q, \, \gamma}(g; \, z) = \sup_{r > 0} r^{-\gamma} |g|_{q, \, B(z, r)},
\]
where $\gamma$ is a positive real number and $B(z, r)$ is the $\rho$-ball centered at $z$ with radius $r$.

Let $k$ a non negative integer and $\mathcal{P}_{k}$ the subspace of $L^{q}_{loc}(\mathbb{H}^{n})$ formed by all the polynomials of homogeneous degree at most $k$. We denote by $E^{q}_{k}$ the quotient space of $L^{q}_{loc}(\mathbb{H}^{n})$ by $\mathcal{P}_{k}$. If 
$G \in E^{q}_{k}$, we define the seminorm $\| G \|_{q, \, B} = \inf \left\{ |g|_{q, \, B} : g \in G \right\}$. The family of all these seminorms induces on $E^{q}_{k}$ the quotient topology.

Given a positive real number $\gamma$, we can write $\gamma = k + t$, where $k$ is a non negative integer and $0 < t \leq 1$. This decomposition is unique.

For $G \in E^{q}_{k}$, we define a maximal function $N_{q, \, \gamma}(G; z)$ as 
\[
N_{q, \, \gamma}(G; z) = \inf \left\{ \eta_{q, \, \gamma}(g; z) : g \in G \right\}.
\]

\begin{lemma} The maximal function $z \to N_{q; \, \gamma}(G; z)$ associated with a class $G$ in 
$E_{k}^{q}$ is lower semicontinuous.
\end{lemma}

\begin{proof}
It is easy to check that $\eta_{q, \gamma}(g; \, \cdot)$ is lower semicontinuous for every $g \in G$ 
(i.e: the set $\{ z : \eta_{q, \gamma}(g; \, z) > \alpha \}$ is open for all $\alpha \in \mathbb{R}$). Then, 
for $z_0 \in \mathbb{H}^n$ we have 
\[
N_{q; \, \gamma}(G; z_0) \leq \eta_{q, \gamma}(g; \, z_0) \leq \liminf_{z \to z_0} \eta_{q, \gamma}(g; \, z) \,\,\,\, \text{for all} \,\, 
g \in G.
\]
So,
\begin{equation} \label{N eta}
N_{q; \, \gamma}(G; z_0) - \epsilon < \liminf_{z \to z_0} \eta_{q, \gamma}(g; \, z), \,\,\,\, \text{for all} \,\, \epsilon > 0 \,\,\, 
\text{and all} \,\, g \in G.
\end{equation}
Suppose $\displaystyle{\liminf_{z \to z_0}} \, N_{q; \, \gamma}(G; z) < N_{q; \, \gamma}(G; z_0)$. Then, there exists $\epsilon >0$ such that
\[
\liminf_{z \to z_0} N_{q; \, \gamma}(G; z) < N_{q; \, \gamma}(G; z_0) - \epsilon.
\]
Thus, there exists $\delta_0 > 0$ such that for every $0 < \delta < \delta_0$ there exist $z \in B(z_0, \delta) \setminus \{ z_0 \}$ and 
$g = g_z \in G$ such that
\[
\eta_{q, \gamma}(g; \, z) \leq N_{q; \, \gamma}(G; z_0) - \epsilon,
\]
which contradicts (\ref{N eta}). So, it must be 
$N_{q; \, \gamma}(G; z_0) \leq \displaystyle{\liminf_{z \to z_0}} \, N_{q; \, \gamma}(G; z)$. Then, the lemma follows.
\end{proof}

\begin{definition}
Let $0 < p < \infty$ be fixed, we say that an element $G \in E^{q}_{k}$ belongs to the Calder\'on-Hardy space
$\mathcal{H}^{p}_{q, \, \gamma}(\mathbb{H}^{n})$ if the maximal function $N_{q, \, \gamma}(G; \, \cdot \,) \in L^{p}(\mathbb{H}^{n})$. 
The "norm" of $G$ in $\mathcal{H}^{p}_{q, \, \gamma}(\mathbb{H}^{n})$ is defined as 
\[
\| G \|_{\mathcal{H}^{p}_{q, \, \gamma}(\mathbb{H}^{n})} = \| N_{q, \, \gamma}(G; \, \cdot \,) \|_{L^p(\mathbb{H}^{n})}.
\]
\end{definition}

\begin{lemma} \label{puntual 1} Let $G \in E^{q}_{k}$ with $N_{q, \, \gamma}(G; z_0) < \infty,$ for some $z_0 \in \mathbb{H}^{n}$. Then:

$(i)$ There exists a unique $g \in G$ such that $\eta_{q, \, \gamma} (g; z_0) < \infty$ and, therefore, 
$\eta_{q, \, \gamma} (g; z_0) = N_{q, \, \gamma}(G; z_0)$.

$(ii)$ For any $\rho$-ball $B$, there is a constant $c$ depending on $z_0$ and $B$ such that if $g$ is the unique representative of $G$ given in $(i)$, then
\[
\|G\|_{q, \, B} \leq |g|_{q, \, B} \leq c \, \eta_{q, \, \gamma} (g; z_0) = c \, N_{q, \, \gamma}(G; z_0).
\]

The constant $c$ can be chosen independently of $z_0$ provided that $z_0$ varies in a compact set.
\end{lemma}

\begin{proof} Since every polynomial of homogeneous degree at most $k$ can be centered at $z_0$, with $z_0$ being an arbitrary point of 
$\mathbb{H}^n$, by the formula that appears in \cite[Section 5.2, p. 272]{Bonfi}) for the Taylor polynomial of a smooth function, it follows that the argument used to prove \cite[Lemma 3]{segovia} works on $\mathbb{H}^n$ as well.
\end{proof}

\begin{corollary} If $\{ G_{j} \}$ is a sequence of elements of $E^{q}_{k}$ converging to $G$ in 
$\mathcal{H}^{p}_{q, \, \gamma}(\mathbb{H}^{n})$, then $\{ G_{j} \}$ converges to $G$ in $E^{q}_{k}$.
\end{corollary}

\begin{proof} For any $\rho$-ball $B$, by $(ii)$ of Lemma \ref{puntual 1}, we have
\[
\| G- G_{j} \|_{q, \, B} \leq c \, \| \chi_{B} \|_{L^p(\mathbb{H}^{n})}^{-1} 
\| \chi_{B} \,\, N_{q, \, \gamma}(G - G_{j}; \, \cdot \,) \|_{L^p(\mathbb{H}^{n})} 
\leq c \, \| G - G_{j} \|_{\mathcal{H}^{p}_{q, \, \gamma}(\mathbb{H}^{n})},
\]
which proves the corollary.
\end{proof}

\begin{lemma} \label{series in Eqk} Let $\{ G_{j} \}$ be a sequence in $E^{q}_{k}$ such that for a given point $z_0 \in \mathbb{H}^n$, the series 
$\sum_j N_{q, \, \gamma}(G_{j}; \, z_0 )$ is finite. Then:

$(i)$ The series $\sum_j G_j$ converges in $E_{k}^{q}$ to an element $G$ and 
\[
N_{q, \, \gamma}(G; \, z_0 ) \leq \sum_j N_{q, \, \gamma}(G_{j}; \, z_0 ).
\]

$(ii)$ If $g_j$ is the unique representative of $G_j$ satisfying
$\eta_{q, \, \gamma} (g_j; z_0) = N_{q, \, \gamma}(G_j; z_0)$, then $\sum_j g_j$ converges in $L^{q}_{loc}(\mathbb{H}^{n})$ to a function 
$g$ that is the unique representative of $G$ satisfying $\eta_{q, \, \gamma} (g; z_0) = N_{q, \, \gamma}(G; z_0)$
\end{lemma}

\begin{proof} The proof is similar to the one given in \cite[Lemma 4]{segovia}.
\end{proof}

\begin{proposition} \label{cerrado} The space $\mathcal{H}^{p}_{q, \, \gamma}(\mathbb{H}^{n})$, $0 < p < \infty$, is complete.
\end{proposition}

\begin{proof} Given $0 < p < \infty$, let $\underline{p} := \min\{ p, 1 \}$. It is enough to show that $\mathcal{H}^{p}_{q, \, \gamma}$ has the Riesz-Fisher property: given any sequence $\{ G_j \}$ in 
$\mathcal{H}^{p}_{q, \, \gamma}$ such that 
\[
\sum_{j} \| G_j \|_{\mathcal{H}^{p}_{q, \, \gamma}}^{\underline{p}} < \infty,
\]
the series $\sum_{j} G_j$ converges in $\mathcal{H}^{p}_{q, \, \gamma}$. \\
Let $m \geq 1$ be fixed, then 
\[
\left\| \sum_{j=m}^{k} N_{q, \, \gamma}(G_{j}; \, \cdot \,) \right\|_{L^p}^{\underline{p}} \leq \sum_{j=m}^{k} \left\| N_{q, \, \gamma}(G_{j}; \, \cdot \,) \right\|_{L^p}^{\underline{p}} \leq \sum_{j=m}^{\infty} \| G_j \|_{\mathcal{H}^{p}_{q, \, \gamma}}^{\underline{p}} 
=: \alpha_m < \infty,
\]
for every $k \geq m$. Thus 
\[
\int_{\mathbb{H}^{n}} \, \left( \alpha_m^{-1/ \underline{p}} \, \sum_{j=m}^{k} N_{q, \, \gamma}(G_{j}; \, z ) \right)^{p} \, dz
\]
\[ 
\leq \int_{\mathbb{H}^{n}} \left( \left\| \sum_{j=m}^{k} N_{q, \, \gamma}(G_{j}; \, \cdot \,) \right\|_{L^p}^{-1} \, \sum_{j=m}^{k} N_{q, \, \gamma}(G_{j};  z ) \right)^{p} \, dz =1, \,\,\, \forall \, k \geq m,
\]
by applying Fatou's lemma as $k \rightarrow \infty$, we obtain
\[
\int_{\mathbb{H}^{n}} \, \left( \alpha_m^{-1/\underline{p}} \, \sum_{j=m}^{\infty} N_{q, \, \gamma}(G_{j}; \, z ) \right)^{p} \, dz \leq 1,
\]
so
\begin{equation}
\left\|  \sum_{j=m}^{\infty} N_{q, \, \gamma}(G_{j}; \, \cdot \,) \right\|_{L^p}^{\underline{p}} \leq \alpha_m = \sum_{j=m}^{\infty} \| G_j \|_{\mathcal{H}^{p}_{q, \, \gamma}}^{\underline{p}} < \infty, \,\,\,\, \forall \, m \geq 1  \label{serie}.
\end{equation}
Taking $m = 1$ in (\ref{serie}), it follows that $\sum_{j} N_{q, \, \gamma}(G_{j}; z)$ is finite a.e. 
$z \in \mathbb{H}^{n}$. Then, by $(i)$ of Lemma \ref{series in Eqk}, the series $\sum_j G_j$ converges in $E_{k}^{q}$ to an element $G$. Now 
\[
N_{q, \, \gamma}\left( G - \sum_{j=1}^{k} G_j; z \right) \leq \sum_{j=k+1}^{\infty} N_{q, \, \gamma} (G_j; z),
\]
from this and (\ref{serie}) we get
\[
\left\| G - \sum_{j=1}^{k} G_j \right\|_{\mathcal{H}^{p}_{q, \, \gamma}}^{\underline{p}} \leq \sum_{j=k+1}^{\infty} 
\| G_j \|_{\mathcal{H}^{p}_{q, \, \gamma}}^{\underline{p}},
\]
and since the right-hand side tends to $0$ as $k \rightarrow \infty$, the series $\sum_{j}G_j$ converges to $G$ in $\mathcal{H}^{p}_{q, \, \gamma}(\mathbb{H}^{n})$.
\end{proof}

\begin{proposition} \label{g distrib}
If $g \in L^{q}_{loc}(\mathbb{H}^{n})$, $1 < q < \infty$, and there is a point $z_0 \in \mathbb{H}^{n}$ such that 
$\eta_{q, \, \gamma} (g ; z_0) < \infty$, then $g \in \mathcal{S}'(\mathbb{H}^{n})$.
\end{proposition}

\begin{proof}
We first assume that $z_0 = e = (0, 0)$. Given $\varphi \in \mathcal{S}(\mathbb{H}^{n})$ and $N > \gamma + Q$ (where $Q=2n+2$), we have 
that $|\varphi(w)| \leq \| \varphi \|_{\mathcal{S}(\mathbb{H}^{n}), \, N} \, (1 + \rho(w))^{-N}$ for all $w \in \mathbb{H}^{n}$. So
\begin{equation*}
\begin{split}
\left| \int_{\mathbb{H}^{n}} g(w) \varphi(w) dw \right| & \leq \| \varphi \|_{\mathcal{S}(\mathbb{H}^{n}), \, N} 
\int_{\rho(w) < 1} |g(w)| (1 + \rho(w))^{-N} dw \\
& + \| \varphi \|_{\mathcal{S}(\mathbb{H}^{n}), \, N} \sum_{j=0}^{\infty} \int_{2^j \leq \rho(w) < 2^{j+1}} |g(w)| 
(1+\rho(w))^{-N} dw \\
& \lesssim  \| \varphi \|_{\mathcal{S}(\mathbb{H}^{n}), \, N} \, \eta_{q, \gamma}(g; e) \\
& +  \| \varphi \|_{\mathcal{S}(\mathbb{H}^{n}), \, N} \, \eta_{q, \gamma}(g; e) \, \sum_{j=0}^{\infty} 2^{j(\gamma + Q - N)},
\end{split}
\end{equation*}
where in the last estimate we use the Jensen's inequality. Since $N > \gamma + Q$ it follows that $g \in \mathcal{S}'(\mathbb{H}^{n})$. For the case $z_0 \neq e$ we apply the translation operator $\tau_{z_0}$ defined by $(\tau_{z_0}g)(z) = g(z_{0}^{-1} \cdot z)$ and use the fact that $\eta_{q, \gamma} \left( \tau_{z_{0}^{-1}} g; \, e \right) = \eta_{q, \gamma}(g; \, z_0)$ (see Remark \ref{cambio de centro}).
\end{proof}

\begin{proposition} \label{Lg dist} Let $g \in L^q_{loc} \cap \mathcal{S}'(\mathbb{H}^n)$ and $f = \mathcal{L} g$ in 
$\mathcal{S}'(\mathbb{H}^n)$. If $\phi \in \mathcal{S}(\mathbb{H}^n)$ and $N > Q+2$, then
\[
(M_{\phi}f)(z) := \sup \left\{ |(f \ast \phi_t)(w)| : \rho(w^{-1} \cdot z) < t, \, 0 < t < \infty \right\}  
\]
\[
\leq C \| \phi \|_{\mathcal{S}(\mathbb{H}^{n}), N}  \,\,\, \eta_{q, 2}(g; \, z)
\]
holds for all $z \in \mathbb{H}^n$.
\end{proposition}

\begin{proof} Let $\rho(w^{-1} \cdot z) < t$, since $f = \mathcal{L} g$ in $\mathcal{S}'(\mathbb{H}^n)$ a computation gives
\[
(f \ast \phi_t)(w) = t^{-2} (g \ast (\mathcal{L} \phi)_t)(w) = t^{-2} \int g(u) (\mathcal{L} \phi)_t (u^{-1} \cdot w) du.
\]
Applying Remark \ref{cambio de centro} and (\ref{homog dim}), we get
\begin{equation} \label{g conv Lphi}
(f \ast \phi_t)(w) = t^{-2} \int g(z \cdot tu) (\mathcal{L} \phi) (u^{-1} \cdot t^{-1}(z^{-1} \cdot w)) du.
\end{equation}
Being $\rho(z^{-1} \cdot w) < t$, a computation gives
\begin{equation} \label{1 mas ro}
1 + \rho(u) \leq 2 \left( 1 + \rho(u^{-1} \cdot t^{-1}(z^{-1} \cdot w)) \right).
\end{equation}
On the other hand, for $N > 2$, we have
\begin{equation} \label{L ineq}
\left|(\mathcal{L} \phi) (u^{-1} \cdot t^{-1}(z^{-1} \cdot w))\right| \left( 1 + \rho(u^{-1} \cdot t^{-1}(z^{-1} \cdot w)) \right)^N \leq 
\| \phi \|_{\mathcal{S}(\mathbb{H}^{n}), N}.
\end{equation}
Now, from (\ref{1 mas ro}) and (\ref{L ineq}), it follows that
\begin{equation} \label{L ineq 2}
\left|(\mathcal{L} \phi) (u^{-1} \cdot t^{-1}(z^{-1} \cdot w))\right| \leq 2^N \| \phi \|_{\mathcal{S}(\mathbb{H}^{n}), N} (1 + \rho(u))^{-N},
\end{equation}
for $\rho(z^{-1} \cdot w) < t$. Then, (\ref{g conv Lphi}), (\ref{L ineq 2}) and (\ref{homog dim}) give
\[
2^{-N} \| \phi \|_{\mathcal{S}(\mathbb{H}^{n}), N}^{-1} \left|(f \ast \phi_t)(w) \right| \leq  t^{-2} \int |g(z \cdot tu)|(1 + \rho(u))^{-N} du.
\]
\[
= t^{-2} t^{-Q} \int |g(z \cdot u)|(1 + \rho(t^{-1} u))^{-N} du
\]
\[
\leq t^{-2} t^{-Q} \int_{\rho(u) < t} |g(z \cdot u)|(1 + \rho(t^{-1} u))^{-N} du 
\]
\[
+ \,\, t^{-2} t^{-Q} \int_{2^j t \leq \rho(u) < 2^{j+1}t} |g(z \cdot u)| \, \rho(t^{-1} u)^{-N} du 
\]
\[
\lesssim \left( 1 + \sum_{j=0}^{\infty} 2^{j(Q+2-N)} \right) \, \eta_{q, 2}(g; z),
\]
for $\rho(z^{-1} \cdot w) < t$. Applying Jensen's inequality and taking $N > Q+2$ in the last inequality the proposition follows.
\end{proof}

\begin{remark} \label{LG definition} We observe that if $G \in \mathcal{H}^{p}_{q, \, 2}(\mathbb{H}^{n})$, then 
$N_{q, \, 2}(G; z_0) < \infty,$ for some $z_0 \in \mathbb{H}^{n}$. By $(i)$ in Lemma \ref{puntual 1} there exists $g \in G$ such that 
$N_{q, \, 2}(G; z_0) = \eta_{q, \, 2}(g; z_0)$; from Proposition \ref{g distrib} it follows that $g  \in \mathcal{S}'(\mathbb{H}^{n})$. So 
$\mathcal{L} g$ is well defined in sense of distributions. On the other hand, since any two representatives of $G$ differ in a polynomial 
of homogeneous degree at most $1$, we get that $\mathcal{L} g$ is independent of the representative $g \in G$ chosen. Therefore, for 
$G \in \mathcal{H}^{p}_{q, \, 2}(\mathbb{H}^{n})$, we define $\mathcal{L} G$ as the distribution $\mathcal{L} g$, where $g$ is any representative of $G$.
\end{remark}

\begin{theorem} \label{L 1-1}
If $G \in \mathcal{H}^{p}_{q, 2}(\mathbb{H}^{n})$ and $\mathcal{L} G = 0$, then $G \equiv 0$.
\end{theorem}

\begin{proof} Let $G \in \mathcal{H}^{p}_{q, 2}(\mathbb{H}^{n})$ and $g \in G$ such that 
$\eta_{q, \, 2}(g; z_0) = N_{q, \, 2}(G; z_0) < \infty$ for some $z_0 \in \mathbb{H}^n \setminus \{ e \}$. If $\mathcal{L} g = 0$, by 
Theorem 2 in \cite{Geller}, we have that $g$ is a polynomial. To conclude the proof it is suffices to show that $g$ is a polynomial of homogeneous degree less than or equal to $1$. Suppose $g(z) = \displaystyle{\sum_{d(I) \leq k} c_I z^I}$, with $k \geq 2$. Then, for 
$\delta \geq 2 \rho(z_0)$
\begin{equation*}
\begin{split}
[\eta_{q, 2}(g; z_0)]^q  \delta^{(2-k)q} & \geq C \delta^{-Q - kq} \int_{\rho(z_0^{-1} \cdot w) < \delta} \left|\sum_{d(I) \leq k} c_I \, 
w^I \right|^q dw \\
& \geq C \delta^{-Q - kq} \int_{\rho(w) < \delta/2} \left|\sum_{d(I) \leq k} c_I \, w^I \right|^q dw \\
& = C 2^{-Q - kq} \int_{\rho(z) < 1} \left|\sum_{d(I) = k} c_I \, z^I \right|^q dz + o_{\delta}(1).
\end{split}
\end{equation*}
Thus if $k > 2$, letting $\delta \to \infty$, we have
\[
\int_{\rho(z) < 1} \left|\sum_{d(I) = k} c_I \, z^I \right| dz = 0,
\]
which implies that $c_I=0$ for $d(I)=k$, contradicting the assumption that $g$ is of homogeneous degree $k$. On the other hand, if $k=2$ letting $\delta \to \infty$ we obtain that
\[
\int_{\rho(z) < 1} \left|\sum_{d(I) = 2} c_I \, z^I \right| dz \lesssim [\eta_{q, 2}(g; z_0)]^q = [N_{q, 2}(G; z_0)]^q.
\]
Since $N_{q, 2}(G; \, \cdot) \in L^p(\mathbb{H}^n)$, to apply Lemma \ref{finite measure} with 
$\mathcal{O} = \{ z : N_{q, 2}(G; \, z) > 1 \}$ and $h = N_{q, 2}(G; \, \cdot)$, the amount $N_{q, 2}(G; z_0)$ can be taken arbitrarily small and so
\[
\int_{\rho(z) < 1} \left|\sum_{d(I) = 2} c_I \, z^I \right| dz = 0,
\]
which contradicts that $g$ is of homogeneous degree $2$. Thus $g$ is a polynomial of homogeneous degree less than or equal to $1$, as we wished to prove. 
\end{proof}

If $a$ is a bounded function with compact support, its potential $b$, defined as 
\[
b(z) := \left( a \ast c_n \, \rho^{-2n} \right)(z) = c_n \int_{\mathbb{H}^{n}} \rho(w^{-1} \cdot z)^{-2n} a(w) dw,
\]
is a locally bounded function and, by Theorem \ref{fund solution}, $\mathcal{L} b = a$ in the sense of distributions. For these potentials, we have the following result.

\

In the sequel, $Q = 2n+2$ and $\beta$ is the constant in \cite[Corollary 1.44]{Foll-St}, we observe that $\beta \geq 1$ 
(see \cite[p. 29]{Foll-St}).

\begin{lemma} \label{a conv ro} Let $a(\cdot)$ be an $(p, p_{0}, N)$ - atom centered at the $\rho$ - ball $B(z_0, \delta)$ 
with $N \geq N_p$. If 
\[
b(z) = \left( a \ast c_n \, \rho^{-2n} \right)(z),
\] 
then, for $\rho(z_0^{-1} z) \geq 2 \beta^{2}\delta$ and every multi-index $I$ there exists a positive constant $C_{I}$ such that 
\[
\left| (X^{I}b)(z) \right| \leq C_{I} \, \delta^{2+Q} |B|^{-\frac{1}{p}} \rho(z_{0}^{-1} \cdot z)^{-Q-d(I)}
\]
holds.
\end{lemma}

\begin{proof} We fix a multiindex $I$, by the left invariance of the operator $X^I$ and Remark \ref{cambio de centro}, we have that
\begin{equation*}
\begin{split}
(X^{I}b)(z) & = c_n \int_{B(z_0, \delta)} \left(X^{I} \rho^{-2n} \right)(w^{-1} \cdot z) \, a(w) dw \\
& = c_n \int_{B(e, \delta)} \left(X^{I} \rho^{-2n} \right)(u^{-1} \cdot z_0^{-1} \cdot z) \, a(z_0 \cdot u) du,
\end{split}
\end{equation*}
for each $z \notin B(z_0, 2\beta^{2}\delta)$. By the condition $a_3)$ of the atom $a(\cdot)$ and Remark \ref{atomo trasladado}, it follows for $z \notin B(z_0, 2\beta^{2}\delta)$ that
\begin{equation} \label{Xb}
(X^{I}b)(z) = c_n \int_{B(e, \delta)} \, \left[ \left(X^{I} \rho^{-2n} \right)(u^{-1} \cdot z_0^{-1} \cdot z) - q(u^{-1}) \right] 
a(z_0 \cdot u) \, du,
\end{equation}
where $u \to q(u^{-1})$ is the right Taylor polynomial at $e$ of homogeneous degree $1$ of the function 
\[
u \to \left(X^{I} \rho^{-2n} \right)(u^{-1} \cdot z_{0}^{-1} \cdot z).
\] 
Then by the right-invariant version of the Taylor inequality in \cite[Corollary 1.44]{Foll-St},
\begin{equation*}
\left| \left(X^{I} \rho^{-2n} \right)(u^{-1} \cdot z_{0}^{-1} \cdot z) - q(u^{-1}) \right| \lesssim \hspace{0.3cm} \rho(u)^{2} \,\,\, \times 
\end{equation*}
\begin{equation} \label{Taylor ineq}
\sup_{\rho(v) \leq \beta^{2}\rho(u), d(J)=2} \left|\left(\widetilde{X}^{J} \left(X^{I} \rho^{-2n} \right) \right)
(v \cdot z_{0}^{-1} \cdot z) \right|.
\end{equation}
Now, for $u \in B(e, \delta)$, $z_{0}^{-1} \cdot z \notin B(e, 2\beta^{2}\delta)$ and $\rho(v) \leq \beta^{2} \rho(u)$, we obtain that
$\rho(z_{0}^{-1} \cdot z) \geq 2\rho(v)$ and hence $\rho(v \cdot z_{0}^{-1} \cdot z) \geq \rho(z_{0}^{-1} \cdot z)/2$, then 
(\ref{Taylor ineq}) and Lemma \ref{ro estimate} with $\alpha = 2n$ and $d(J)=2$ allow us to get
\[
\left| \left(X^{I} \rho^{-2n} \right) (u^{-1} \cdot z_{0}^{-1} \cdot z) - q(u^{-1}) \right| \lesssim 
\delta^{2} \rho(z_{0}^{-1} \cdot z)^{-2n-2-d(I)}.
\]
This estimate, (\ref{Xb}), and the conditions $a_1)$ and $a_2)$ of the atom $a(\cdot)$ lead to
\begin{equation*}
\begin{split}
\left| (X^{I}b)(z) \right| & \lesssim \delta^{2} \rho(z_{0}^{-1} \cdot z)^{-2n-2-d(I)} \| a \|_{L^{1}(\mathbb{H}^n)} \\
& \lesssim \delta^{2} \rho(z_{0}^{-1} \cdot z)^{-2n-2-d(I)} |B|^{1-\frac{1}{p_0}} \| a \|_{L^{p_0}(\mathbb{H}^n)} \\
& \lesssim \delta^{2} \rho(z_{0}^{-1} \cdot z)^{-2n-2-d(I)} |B|^{1-\frac{1}{p}} \\
& \lesssim \delta^{2+Q} |B|^{-\frac{1}{p}} \rho(z_{0}^{-1} \cdot z)^{-Q-d(I)},
\end{split}
\end{equation*}
for $\rho(z_0^{-1} \cdot z) \geq 2 \beta^{2}\delta$. This concludes the proof.
\end{proof}

The following result is crucial to get Theorem \ref{principal thm}.

\begin{proposition} \label{pointwise estimate} Let $a(\cdot)$ be an $(p, p_{0}, N)$ - atom centered at the $\rho$ - ball 
$B=B(z_0, \delta)$. If $b(z) = (a \ast c_n \rho^{-2n})(z)$, then for all $z \in \mathbb{H}^{n}$
\begin{equation} \label{N estimate}
\begin{split}
N_{q, 2} \left(\widetilde{b} \, ; z \right) & \lesssim | B|^{-1/p} \left[(M \chi_{B})(z) \right]^{\frac{2 + Q/q}{Q}} + 
\chi_{4 \beta^2 B}(z) (M a)(z)  \\
& + \chi_{4 \beta^2 B}(z) \sum_{d(I)=2} (T^{*}_{I} a)(z),
\end{split}
\end{equation}
where $\widetilde{b}$ is the class of $b$ in $E^{q}_{1}$, $M$ is the Hardy-Littlewood maximal operator and $(T^{*}_{I} a) (z) = \sup_{\epsilon >0} \left|\int_{\rho(w^{-1} \cdot z) > \epsilon} \, (X^{I} \rho^{-2n})(w^{-1} \cdot z) a(w) \, dw \right|$.
\end{proposition}

\begin{proof} For an atom $a(\cdot)$ satisfying the hypothesis of Proposition, we set
\[
R(z, w)= b(z \cdot w) - \sum_{0 \leq d(I) \leq 1} (X^{I}b)(z) w^{I}
\]
\[
=b(z \cdot w) - \sum_{0 \leq d(I) \leq 1} \left[\int_{B(z_0, \delta)} \, (X^{I} c_n\rho^{-2n})(u^{-1} \cdot z) a(u) \ du \right] 
w^I,
\]
where $w \to \sum (X^{I}b)(z) w^{I}$ is the left Taylor polynomial of the function $w \to b(z \cdot w)$ at $w=e$ of homogeneous degree $1$ 
(see \cite{Bonfi}, p. 272). We observe that if $I=(i_1, ..., i_{2n}, i_{2n+1})$ is a multi-index such that 
$d(I) \leq 1$, then $i_{2n+1} = 0$.

\

Next, we shall estimate $|R(z,w)|$ considering the cases 
\[
\rho(z_0^{-1} \cdot z) \geq 4 \beta^2 \delta \,\,\,\,\,\,\, \text{and} \,\,\,\,\,\,\, \rho(z_0^{-1} \cdot z) < 4 \beta^2 \delta
\]
separately, and then we will obtain the estimate (\ref{N estimate}).

\

\textbf{Case:} $\rho(z_0^{-1} \cdot z) \geq 4 \beta^2 \delta$.

\

For $\rho(z_0^{-1} \cdot z) \geq 4 \beta^2 \delta$, $\rho(w) \leq \frac{1}{2 \beta^2} \rho(z_0^{-1} \cdot z)$ and 
$\rho(u) \leq \beta^2 \rho(w)$, a computation gives $\rho(z_0^{-1} \cdot z \cdot u) \geq 2 \beta^2 \delta$. Then, by the left-invariant Taylor inequality in \cite[Corollary 1.44]{Foll-St} and Lemma \ref{a conv ro}, we get
\begin{equation*}
\begin{split}
|R(z, w)| & \lesssim \rho(w)^2 \sup_{\rho(u) \leq \beta^{2}\rho(w), \, d(I)=2} \left|(X^{I} b)( z \cdot u) \right| \\
& \lesssim |B|^{-1/p} \left(\frac{\delta}{\rho(z_0^{-1} \cdot z)}\right)^{2+Q} \rho(w)^{2}.  \label{R1}
\end{split}
\end{equation*}

Now, let $\rho(w) \geq \frac{1}{2 \beta^2} \rho(z_0^{-1} \cdot z)$. We have
\[
|R(z, w)| \leq |b(z \cdot w)| + \sum_{0 \leq d(I) \leq 1} |(X^{I}b)(z)| |w^{I}|.
\]
Since $\rho(z_0^{-1} \cdot z) \geq 4 \beta^2 \delta$, by Lemma \ref{a conv ro} and observing that 
$\rho(w)/\rho(z_0^{-1} \cdot z) > \frac{1}{2 \beta^2}$, 
we have
\[
|(X^{I}b)(z)| |w^{I}| \lesssim |B|^{-1/p} \left( \frac{\delta}{\rho(z_0^{-1} \cdot z)} \right)^{2+Q} \rho(w)^{2}.
\]
As for the other term, $|b(z \cdot w)|$, we consider separately the cases 
\[
\rho(z_0^{-1} \cdot z \cdot w) > 2 \beta^2 \delta \,\,\,\,\, \text{and} \,\,\,\,\, \rho(z_0^{-1} \cdot z \cdot w) \leq 2 \beta^2 \delta. 
\]
In the case $\rho(z_0^{-1} \cdot z \cdot w) > 2 \beta^2 \delta$, 
we apply Lemma \ref{a conv ro} with $I = 0$, obtaining
\[
|b(z\cdot w)| \lesssim |B|^{-1/p} \delta^{2+Q} \rho(z_0^{-1} \cdot z \cdot w)^{-Q}.
\]
Then
\begin{equation} \label{R2}
|R(z,w)| \lesssim |B|^{-1/p} \delta^{2+Q} \rho(z_0^{-1} \cdot z \cdot w)^{-Q} + 
|B|^{-1/p} \left( \frac{\delta}{\rho(z_0^{-1} \cdot z)} \right)^{2+Q} \rho(w)^{2}
\end{equation}
holds if $\rho(z_0^{-1} \cdot z) > 4 \beta^2 \delta$, $\rho(w) \geq \frac{1}{2 \beta^2} \rho(z_0^{-1} \cdot z)$ and 
$\rho(z_0^{-1} \cdot z \cdot w) > 2 \beta^2 \delta$.

For $\rho(z_0^{-1} \cdot z \cdot w) \leq 2 \beta^2 \delta$, we have 
$B(z_0, \delta) \subset \{ u : \rho(u^{-1} \cdot z \cdot w) < (1+2 \beta^2) \delta \} =: \Omega_{\delta}$, so
\begin{equation*}
\begin{split}
|b(z\cdot w)| & = c_n \left| \int_{B(z_0, \delta)} \rho(u^{-1} \cdot z \cdot w)^{-2n} a(u) du \right| \\
& \lesssim \|a \|_{L^{p_0}} \left(\int_{B(z_0, \delta)} \rho(u^{-1} \cdot z \cdot w)^{-2np'_0} du \right)^{1/ p'_{0}} \\
& \lesssim \|a \|_{L^{p_0}} \left(\int_{\Omega_\delta} \rho(u^{-1} \cdot z \cdot w)^{-2np'_0} du \right)^{1/ p'_{0}}.
\end{split}
\end{equation*}
Since $a(\cdot)$ is an $(p, p_0, N)$ - atom, we can choose $p_0 > Q/2$, and get
\begin{equation*}
\begin{split}
|b(z \cdot w)| & \lesssim |B|^{-1/p} \delta^{Q/p_0} \left( \int_{0}^{(1+2\beta^2)\delta}  r^{-2n p'_0 + Q -1} dr \right)^{1/p'_0} \\
& \lesssim |B|^{-1/p} \delta^{Q/p_0} \delta^{-2n} \delta^{Q/p'_0} = |B|^{-1/p} \delta^2.
\end{split}
\end{equation*}

Since $\rho(z_0^{-1} \cdot z) \geq 4 \beta^2 \delta$ we can conclude that
\begin{equation} \label{R3}
|R(z, w)| \lesssim  |B|^{-1/p} \delta^2 + |B|^{-1/p} \left( \frac{\delta}{\rho(z_0^{-1} \cdot z)} \right)^{2+Q} \rho(w)^{2}, 
\end{equation}
for all $|\rho(w)| \geq \frac{1}{2 \beta^2} \rho(z_0^{-1} z)$ and $\rho(z_0^{-1} \cdot z \cdot w) \leq 2 \beta^2 \delta$.

Let us the estimate
\[
r^{-2} \left( |B(e, r)|^{-1} \int_{B(e, r)} |R(z, w)|^{q} dw \right)^{1/q}, \,\,\,\, r > 0.
\]
For them, we split the domain of integration into three subsets:

\qquad

$D_1 = \left\{ w \in B(e, r) : \rho(w) \leq \frac{1}{2 \beta^2} \rho(z_0^{-1} \cdot z) \right\}$,

\qquad

$D_2 = \left\{ w \in B(e, r) : \rho(w) \geq \frac{1}{2 \beta^2} \rho(z_0^{-1} \cdot z), \, 
\rho(z_0^{-1} \cdot z \cdot w) > 2 \beta^2 \delta \right\},$
\\\\
and

\qquad

$D_3 = \left\{w \in B(e, r) : \rho(w) \geq \frac{1}{2 \beta^2} \rho(z_0^{-1} \cdot z), \, 
\rho(z_0^{-1} \cdot z \cdot w) \leq 2 \beta^2 \delta \right\}$
\\\\
According to the estimates obtained for $|R(z,w)|$ above, we use on $D_1$ the estimate (\ref{R1}), on $D_2$ the estimate (\ref{R2}) and on $D_3$ the estimate (\ref{R3}) to get
\[
r^{-2} \left( |B(e, r)|^{-1} \int_{B(e, r)} |R(z, w)|^{q} dw \right)^{1/q} \lesssim |B |^{-1/p} 
\left(\frac{\delta}{\rho(z_0^{-1} \cdot z)} \right)^{2+Q/q}.
\]
Thus,
\begin{equation}
N_{q, 2}\left(\widetilde{b} \, ; z \right) \lesssim |B |^{-1/p}  M(\chi_{B})(z)^{\frac{2+Q/q}{Q}}, \label{Nq3}
\end{equation}
if $\rho(z_0^{-1} \cdot z) \geq 4 \beta^2 \delta$.

\qquad

\textbf{Case:} $\rho(z_0^{-1} \cdot z) < 4 \beta^2 \delta$.

\qquad

We have
\[
R(z, w) = c_n \int \left[ \rho^{-2n}(u^{-1} \cdot z \cdot w) - \sum_{0 \leq d(I) \leq 1} (X^I \rho^{-2n})(u^{-1} \cdot z) w^{I} \right] 
a(u) du
\]
\[
=\int_{\rho(u^{-1} \cdot z) < 2 \beta^2 \rho(w)} \, + \int_{\rho(u^{-1} \cdot z) \geq 2 \beta^2 \rho(w)} \, = J_1(z, w) + J_2(z, w).
\]
Assuming that $u \neq z \cdot w$ and $u \neq z$, we can write
\[
U = \rho^{-2n}(u^{-1} \cdot z \cdot w) - \rho^{-2n}(u^{-1} \cdot z) - \sum_{d(I) = 1} (X^I \rho^{-2n})(u^{-1} \cdot z) w^{I}.
\]
By Lemma \ref{ro estimate}, we get
\[
|U| \lesssim \rho(u^{-1} \cdot z \cdot w)^{-2n} + \rho(u^{-1} \cdot z)^{-2n} + \rho(w) \, \rho(u^{-1} \cdot z)^{-2n-1}
\]
Observing that $\rho(u^{-1} \cdot z) < 2 \beta^2 \rho(w)$ implies $\rho(u^{-1} \cdot z \cdot w) < 3 \beta^2 \rho(w)$, we obtain
\[
|J_1(z, w)| \leq \int_{\rho(u^{-1} \cdot z) < 2 \beta^2 \rho(w)} |U||a(u)| du
\]
\[
\lesssim  \int_{\rho(u^{-1} \cdot z \cdot w) < 3 \beta^2 \rho(w)} \rho(u^{-1} \cdot z \cdot w)^{-2n} |a(u)| du  
\]
\[
+
\int_{\rho(u^{-1} \cdot z) < 2 \beta^2 \rho(w)} \rho(u^{-1} \cdot z)^{-2n}  |a(u)| du
\]
\[
+ \rho(w) \int_{\rho(u^{-1} \cdot z) < 2 \beta^2 \rho(w)} \rho(u^{-1} \cdot z)^{-2n-1} |a(u)| du
\]
\[
= \sum_{k=0}^{\infty}\int_{3^{-k}\beta^2\rho(w) \leq \rho(u^{-1} \cdot z \cdot w) <3^{-(k-1)}\beta^2 \rho(w)} 
\rho(u^{-1} \cdot z \cdot w)^{-2n} |a(u)| du
\]
\[
+
\sum_{k=0}^{\infty} \int_{2^{-k}\beta^2 \rho(w) \leq \rho(u^{-1} \cdot z) < 2^{-(k-1)}\beta^2 \rho(w)} \rho(u^{-1} \cdot z)^{-2n}  |a(u)| du
\]
\[
+ \rho(w) \sum_{k=0}^{\infty} \int_{2^{-k} \beta^2 \rho(w) \leq \rho(u^{-1} \cdot z) < 2^{-(k-1)} \beta^2 \rho(w)} 
\rho(u^{-1} \cdot z)^{-2n-1} |a(u)| du
\]
\[
\lesssim \rho(w)^2 (Ma)(z).
\]
To estimate $J_2(z, w)$, we can write (see \cite{Bonfi}, p. 272, taking into account that $x^t J x = 0$ for all $x \in \mathbb{R}^{2n}$)
\[
U= \left[\rho^{-2n}(u^{-1} \cdot z \cdot w) - \sum_{d(I) \leq 2} (X^I \rho^{-2n})(u^{-1} \cdot z) \frac{w^{I}}{|I|!} \right] + 
\sum_{d(I) = 2} (X^I \rho^{-2n})(u^{-1} \cdot z) \frac{w^{I}}{|I|!}
\]
\[
=U_1+U_2.
\]
For $\rho(u^{-1} \cdot z) \geq 2 \beta^2 \rho(w)$ and $\rho(\nu) \leq \beta^2 \rho(w)$, we have  
$\rho(u^{-1} \cdot z \cdot \nu) \geq \rho(u^{-1} \cdot z)/2$. Then, by the left-invariant Taylor inequality in \cite[Corollary 1.44]{Foll-St} and Lemma \ref{ro estimate}, we get
\begin{equation*}
\begin{split}
|U_1| & \lesssim \rho(w)^3 \sup_{\rho(\nu) \leq \beta^{2}\rho(w), \, d(I)=3} \left|(X^{I} \rho^{-2n})( u^{-1} \cdot z \cdot \nu) \right| \\
& \lesssim \rho(w)^3 \rho(u^{-1} \cdot z)^{-2n-3}.
\end{split}
\end{equation*}
Therefore,
\begin{equation*}
\begin{split}
|J_2(z, w)| & \lesssim \rho(w)^{3} \int_{\rho(u^{-1} \cdot z) \geq 2 \beta^2 \rho(w)} \rho(u^{-1} \cdot z)^{-2n-3} |a(u)| du \\
& + \left| \int_{\rho(u^{-1} \cdot z) \geq 2 \beta^2 \rho(w)} U_2 \, a(u) du \right| \\
& \lesssim \rho(w)^{2} \left( (Ma)(z) + \sum_{d(I)=2} (T_{I}^{\ast}a)(z) \right),
\end{split}
\end{equation*}
where $(T_{I}^{\ast}a)(z) = \sup_{\epsilon >0} \left|\int_{\rho(u^{-1} \cdot z)> \epsilon} \, (X^I \rho^{-2n})(u^{-1} \cdot z) a(u) \, du\right|$.

Now, it is easy to check that 
\[
r^{-2} \left( |B(e, r)|^{-1} \int_{B(e, r)} |J_1(z, w)|^{q} dw \right)^{1/q} \lesssim (Ma)(z)
\]
and
\[
r^{-2} \left( |B(e, r)|^{-1} \int_{B(e, r)} |J_2(z,w)|^{q} dw \right)^{1/q} \lesssim (Ma)(z) + \sum_{d(I)=2} (T^{*}_{I}a)(z).
\]
So
\[
r^{-2} \left( |B(e, r)|^{-1} \int_{B(e, r)} |R(z,w)|^{q} dw \right)^{1/q} \lesssim (Ma)(z) + \sum_{d(I)=2} (T^{*}_{I}a)(z).
\]
This estimate is global, in particular we have that
\begin{equation}
N_{q,2} \left(\widetilde{b} \, ; z \right) \lesssim (Ma)(z) + \sum_{d(I)=2} (T^{*}_{I}a)(z), \label{Nq4}
\end{equation}
for $\rho(z_0^{-1} \cdot z) < 4 \beta^2 \delta$. Finally, the estimates (\ref{Nq3}) and (\ref{Nq4}) for 
$N_{q,2} \left(\widetilde{b} \, ; z \right)$ allow us to obtain (\ref{N estimate}).
\end{proof}

\section{Main results}

We are now in a position to prove our main results.

\begin{theorem} \label{principal thm} Let $Q=2n+2$, $1 < q < \frac{n+1}{n}$ and $Q \, (2 + \frac{Q}{q})^{-1} < p \leq 1$. Then 
the sub-Laplacian $\mathcal{L}$ on $\mathbb{H}^n$ is a bijective mapping from $\mathcal{H}^{p}_{q, 2}(\mathbb{H}^{n})$ onto 
$H^{p}(\mathbb{H}^{n})$. Moreover, there exist two positive constant $c_1$ and $c_2$ such that
\begin{equation} \label{doble ineq}
c_1 \|G \|_{\mathcal{H}^{p}_{q, 2}(\mathbb{H}^{n})} \leq \| \mathcal{L}G \|_{H^{p}(\mathbb{H}^{n})} \leq 
c_2 \|G \|_{\mathcal{H}^{p}_{q, 2}(\mathbb{H}^{n})}
\end{equation}
hold for all $G \in \mathcal{H}^{p}_{q, 2}(\mathbb{H}^{n})$.
\end{theorem}

\begin{proof} The injectivity of the sublaplacion $\mathcal{L}$ in $\mathcal{H}^{p}_{q, 2}(\mathbb{H}^{n})$ was proved in 
Theorem \ref{L 1-1}. 

Let $G \in \mathcal{H}^{p}_{q, \, 2}(\mathbb{H}^{n})$, since $N_{q, 2}(G; z)$ is finite $\text{a.e.} \,\, z \in \mathbb{H}^{n}$, by $(i)$ in 
Lemma \ref{puntual 1} and Proposition \ref{g distrib} the unique representative $g$ of $G$ (which depends on $z$), satisfying 
$\eta_{q, 2}(g; z)=N_{q, 2}(G; z)$, is a function in $L^{q}_{loc}(\mathbb{H}^{n}) \cap \mathcal{S}'(\mathbb{H}^{n})$. In particular, for
a \textit{commutative approximate identity} \footnote{A commutative approximate identity is a function $\phi \in \mathcal{S}(\mathbb{H}^{n})$ such that $\int \phi(z) \, dz = 1$ and $\phi_s \ast \phi_t = \phi_t \ast \phi_s$ for all $s,t > 0$.} $\phi$, by Remark \ref{LG definition} and Proposition \ref{Lg dist} we get
\[
M_{\phi}(\mathcal{L}G)(z) \leq C_{\phi} \,\, N_{q, \, 2}(G; z).
\]
Then, this inequality and Corollary 4.17 in \cite{Foll-St} give $\mathcal{L}G \in H^{p}(\mathbb{H}^{n})$ and 
\begin{equation} \label{continuity}
\| \mathcal{L}G \|_{H^{p}(\mathbb{H}^{n})} \leq C \, \| G \|_{\mathcal{H}^{p}_{q, \, 2}(\mathbb{H}^{n})}.
\end{equation}
This proves the continuity of sub-Laplacian $\mathcal{L}$ from $\mathcal{H}^{p}_{q, \, 2}(\mathbb{H}^{n})$ into $H^{p}(\mathbb{H}^{n})$.

Now we shall see that the operator $\mathcal{L}$ is onto. Given $f \in H^{p}(\mathbb{H}^{n})$, there exist a sequence of nonnegative numbers 
$\{ k_j \}_{j=1}^{\infty}$ and a sequence of $\rho$ - balls $\{B_j \}_{j=1}^{\infty}$ and $(p, p_0, N)$ atoms $a_j$ supported on $B_j$, 
such that $f= \sum_{j=1}^{\infty} k_j a_j$ and
\begin{equation} \label{atomic ineq}
\sum_{j=1}^{\infty} k_{j}^{p} \lesssim \|f \|_{H^{p}(\mathbb{H}^{n})}^{p}.
\end{equation}
For each $j \in \mathbb{N}$ we put $b_j(z)= (a_j \ast c_n \rho^{-2n})(z) = \int_{\mathbb{H}^{n}} c_n \rho(w^{-1} \cdot z)^{-2n} a_j(w) dw$, from Proposition \ref{pointwise estimate} we have
\[
N_{q, 2} \left(\widetilde{b}_j; \, z \right) \lesssim | B_j|^{-1/p} \left[(M \chi_{B_j})(z) \right]^{\frac{2 + Q/q}{Q}} + 
\chi_{4 \beta^2 B_j}(z) (M a_j)(z) 
\]
\[
+ \chi_{4 \beta^2 B_j}(z) \sum_{d(I)=2} (T^{*}_{I} a_j)(z),
\]
so
\begin{equation*}
\begin{split}
\sum_{j=1}^{\infty} k_j N_{q, 2} \left(\widetilde{b}_j; \, z \right) & \lesssim \sum_{j=1}^{\infty} k_j | B_j|^{-1/p} 
\left[(M \chi_{B_j})(z) \right]^{\frac{2 + Q/q}{Q}} \\
& + \sum_{j=1}^{\infty} k_j \chi_{4 \beta^2 B_j}(z) (M a_j)(z) \\
& + \sum_{j=1}^{\infty} k_j \chi_{4 \beta^2 B_j}(z) \sum_{d(I)=2} (T^{*}_{I} a_j)(z) \\
& = I + II + III.
\end{split}
\end{equation*}
To study $I$, by hypothesis, we have that $0 < p \leq 1$ and $(2 + Q/q) p > Q$. Then
\begin{equation*}
\begin{split}
\|I\|_{L^{p}(\mathbb{H}^{n})} & = \left\| \sum_{j=1}^{\infty} k_j |B_j|^{-1/p} M(\chi_{B_j})(\cdot)^{\frac{2 + Q/q}{Q}} 
\right\|_{L^{p}(\mathbb{H}^{n})} \\
& = \left\| \left\{ \sum_{j=1}^{\infty} k_j |B_j|^{-1/p} M(\chi_{B_j})(\cdot)^{\frac{2 + Q/q}{Q}} \right\}^{\frac{Q}{2 + Q/q}} 
\right\|_{L^{\frac{2 + Q/q}{Q} p}(\mathbb{H}^{n})}^{\frac{2 + Q/q}{Q}} \\
& \lesssim \left\| \left\{ \sum_{j=1}^{\infty} k_j |B_j|^{-1/p} \chi_{B_j}(\cdot)  \right\}^{\frac{Q}{2 + Q/q}} 
\right\|_{L^{\frac{2 + Q/q}{Q} p}(\mathbb{H}^{n})}^{\frac{2 + Q/q}{Q}} \\
& = \left\| \sum_{j=1}^{\infty} k_j |B_j|^{-1/p} \chi_{B_j}(\cdot) \right\|_{L^{p}(\mathbb{H}^{n})} \\
& \lesssim \left( \sum_{j=1}^{\infty} k_{j}^{p} \right)^{1/p} \lesssim \|f \|_{H^{p}(\mathbb{H}^{n})},
\end{split}
\end{equation*}
where the first inequality follows from \cite[Theorem 1.2]{Grafakos}, the condition $0 < p \leq 1$ gives the second inequality, and 
(\ref{atomic ineq}) gives the last one.

To study $II$, since $p \leq 1$ we have that
\begin{equation*}
\begin{split}
\| II \|_{L^{p}(\mathbb{H}^n)}^p & \lesssim \left\| \sum_{j} k_j \, \chi_{4\beta^2B_{j}} \, (Ma_j)(\cdot) 
\right\|_{L^{p}(\mathbb{H}^n)}^p \\
& \lesssim \sum_{j} k_j^p \int \chi_{4\beta^2B_{j}}(z) \, (Ma_j)^p(z) \, dz, 
\end{split}
\end{equation*}
applying Holder's inequality with $\frac{p_0}{p}$, using that the maximal operator $M$ is bounded on $L^{p_0}(\mathbb{H}^n)$ and that every
$a_j(\cdot)$ is an $(p, p_0, N)$ - atom,  we get
\begin{equation*}
\begin{split}
\| II \|_{L^{p}(\mathbb{H}^n)}^p & \lesssim \sum_{j} k_j^p |B_j|^{1-\frac{p}{p_0}} 
\left( \int (Ma_j)^{p_0}(z) \, dz \right)^{\frac{p}{p_0}} \\
& \lesssim \sum_{j} k_j^p |B_j|^{1-\frac{p}{p_0}} \| a_j \|_{L^{p_0}(\mathbb{H}^n)}^p \\
& \lesssim \sum_{j} k_j^p |B_j|^{1-\frac{p}{p_0}} |B_j|^{\frac{p}{p_0} - 1} \\
& = \sum_{j} k_j^p \lesssim \|f \|_{H^{p}(\mathbb{H}^{n})}^{p},
\end{split}
\end{equation*}
where the last inequality follows from (\ref{atomic ineq})

To study $III$, by Theorem 3 in \cite{Folland} and Corollary 2, p. 36, in \cite{Elias} (see also \textbf{2.5}, p. 11, in \cite{Elias}), 
we have, for every multi-index $I$ with $d(I) = 2$, that the operator $T_{I}^{*}$ is bounded on $L^{p_0}(\mathbb{H}^n)$ for each 
$1 < p_0 < \infty$. Proceeding as in the estimate of $II$, we get
\[
\| III\|_{L^{p}(\mathbb{H}^n)} \lesssim \left( \sum_{j=1}^{\infty} k_{j}^{p} \right)^{1/p} \lesssim \|f \|_{H^{p}(\mathbb{H}^{n})}.
\]
Thus,
\[
\left\| \sum_{j=1}^{\infty} k_j N_{q, 2} \left( \widetilde{b}_j; \, \cdot \right) \right\|_{L^{p}(\mathbb{H}^n)} \lesssim 
\|f \|_{{H^{p}(\mathbb{H}^n)}}.
\]
Then,
\begin{equation}
\sum_{j=1}^{\infty} k_j N_{q, 2} \left( \widetilde{b}_j; \, z \right) < \infty \,\,\,\,\,\, \text{a.e.} \, z \in \mathbb{H}^{n} \label{Nq}
\end{equation}
and
\begin{equation}
\left\| \sum_{j=M+1}^{\infty} k_j N_{q, 2} \left(\widetilde{b}_j; \cdot \right) \right\|_{L^{p}(\mathbb{H}^n)} \rightarrow 0, \,\,\,\, 
\text{as} \,\,  M \rightarrow \infty  \label{Nq2}.
\end{equation}
From (\ref{Nq}) and Lemma \ref{series in Eqk}, there exists a function $G$ such that $\sum_{j=1}^{\infty} k_j \widetilde{b}_j = G$ 
in $E^{q}_{1}$ and
\[
N_{q, 2} \left( \left(G - \sum_{j=1}^{M} k_j \widetilde{b}_j \right) ; \, z \right) \leq 
C \, \sum_{j=M+1}^{\infty} k_j N_{q, 2}(\widetilde{b}_j; z).
\]
This estimate together with (\ref{Nq2}) implies
\[
\left\| G - \sum_{j=1}^{M} k_j \widetilde{b}_j \right\|_{\mathcal{H}^{p}_{q,2}(\mathbb{H}^n)} \rightarrow 0, \,\,\,\, 
\text{as} \,\,  M \rightarrow \infty.
\]
By proposition \ref{cerrado}, we have that $G \in \mathcal{H}^{p}_{q,2}(\mathbb{H}^{n})$ and $G = \sum_{j=1}^{\infty} k_j \widetilde{b}_j$ in 
$\mathcal{H}^{p}_{q,2}(\mathbb{H}^{n})$. Since $\mathcal{L}$ is a continuous operator from $\mathcal{H}^{p}_{q,2}(\mathbb{H}^{n})$ into 
$H^{p}(\mathbb{H}^{n})$, we get
\[
\mathcal{L}G = \sum_j k_j \mathcal{L} \widetilde{b}_j = \sum_j k_j a_j = f,
\]
in $H^{p}(\mathbb{H}^{n})$. This shows that $\mathcal{L}$ is onto $H^{p}(\mathbb{H}^{n})$. Moreover,
\begin{equation} \label{norm equiv}
\| G\|_{\mathcal{H}^{p}_{q,2}(\mathbb{H}^n)} = 
\left\| \sum_{j=1}^{\infty} k_j \widetilde{b}_j \right\|_{\mathcal{H}^{p}_{q,2}(\mathbb{H}^n)} \lesssim 
\left\| \sum_{j=1}^{\infty} k_j N_{q, 2}(\widetilde{b}_j; \, \cdot) \right\|_{L^{p}(\mathbb{H}^n)}
\end{equation}
\[
\lesssim \|f \|_{H^{p}(\mathbb{H}^n)} 
= \| \mathcal{L} G \|_{H^{p}(\mathbb{H}^n)}.
\]
Finally, (\ref{continuity}) and (\ref{norm equiv}) give (\ref{doble ineq}), and so the proof is concluded.
\end{proof}

Therefore, Theorem \ref{principal thm} allows us to conclude, for $Q(2+Q/q)^{-1} < p \leq 1$, that the equation
\[ 
\mathcal{L} F = f, \,\,\,\,\,\, f \in H^p(\mathbb{H}^n) 
\]
has a unique solution in $\mathcal{H}^{p}_{q, 2}(\mathbb{H}^{n})$, namely: $F := \mathcal{L}^{-1}f$.

\

We shall now see that the case $0 < p \leq Q \, (2 + \frac{Q}{q})^{-1}$ is trivial.

\begin{theorem} \label{2nd thm}
If \, $1 < q < \frac{n+1}{n}$ and $0 < p \leq Q \, (2 + \frac{Q}{q})^{-1}$, then $\mathcal{H}^{p}_{q, \, 2}(\mathbb{H}^{n}) = \{ 0 \}.$
\end{theorem}

\begin{proof} Let $G \in \mathcal{H}^{p}_{q, \, 2}(\mathbb{H}^{n})$ and assume $G \neq 0$. Then there exists $g \in G$ 
that is not a polynomial of homogeneous degree less or equal to $1$. It is easy to check that there exist a positive constant $c$ and 
a $\rho$ - ball $B = B(e, r)$ with $r > 1$ such that
\[
\int_{B} |g(w) - P(w)|^{q} \, dw \geq c > 0,
\]
for every $P \in \mathcal{P}_{1}$.

Let $z$ be a point such that $\rho(z) > r$ and let $\delta = 2 \rho(z)$. Then $B(e, r) \subset B(z, \delta)$. If $h \in G$, then 
$h = g - P$ for some $P \in \mathcal{P}_{1}$ and
\[
\delta^{-2}|h|_{q, B(z, \delta)} \geq c \rho(z)^{-2-Q/q}.
\]
So $N_{q,2}(G; \, z) \geq c \, \rho(z)^{-2-Q/q}$, for $\rho(z) > r$. Since $p \leq Q(2+Q/q)^{-1}$, we have that
\[
\int_{\mathbb{H}^n} [N_{q,2}(G; z)]^p dz \geq c \, \int_{\rho(z) > r} \rho(z)^{-(2+Q/q)p} \, dz = \infty,
\]
which gives a contradiction. Thus $\mathcal{H}^{p}_{q, 2}(\mathbb{H}^{n}) = \{0\}$, if $p \leq Q(2+Q/q)^{-1}$.
\end{proof}

\noindent\textbf{Acknowledgements}\\
\textit{My thanks go to the referee for the useful suggestions and comments which helped me to improve the original manuscript.}

 \bigskip

\noindent Pablo Rocha \\
pablo.rocha@uns.edu.ar \bigskip

\noindent {\small
\noindent Universidad Nacional del Sur \\
Departamento de Matem\'atica \\
Avenida Alem 1253. 2do Piso \\
8000 Bah\'ia Blanca, Argentina
}\bigskip

\end{document}